\documentclass[11pt]{amsart}

\usepackage{amsmath}
\usepackage{amsfonts}
\usepackage{amssymb}
\usepackage{amscd}
\usepackage{amsthm}
\usepackage{framed}
\usepackage{fullpage}
\usepackage{graphicx}
\usepackage{latexsym}
\usepackage[numbers]{natbib}  

\usepackage{hyperref}
\hypersetup{colorlinks,linkcolor={red},citecolor={blue},urlcolor={blue}}

\newtheorem{theorem}{Theorem}[section]
\newtheorem{lemma}[theorem]{Lemma}

\newtheorem{conjecture}[theorem]{Conjecture}

\theoremstyle{definition}

\newtheorem{example}[theorem]{Example}
\newtheorem{remark}[theorem]{Remark}

\numberwithin{equation}{section}

\newcommand{\CC}{\mathbb C}
\newcommand{\HH}{\mathbb H}

\newcommand{\cD}{\mathcal D}

\newcommand{\PP}{\mathbb P}
\newcommand{\QQ}{\mathbb Q}
\newcommand{\RR}{\mathbb R}
\newcommand{\ZZ}{\mathbb Z}

\newcommand{\SL}{\mathop{\mathrm {SL}}\nolimits}

\newcommand{\Sp}{\mathop{\mathrm {Sp}}\nolimits}
\newcommand{\Orth}{\mathop{\null\mathrm {O}}\nolimits}

\newcommand{\rank}{\mathop{\mathrm {rk}}\nolimits}
\newcommand{\latt}[1]{{\langle{#1}\rangle}}
\newcommand{\ord}{\mathop{\mathrm {ord}}\nolimits}

\newcommand{\II}{\mathop{\mathrm {II}}\nolimits}

\def\Borch{\mathbf{B}}

\begin{document}

\title[On the classification of reflective modular forms]{On the classification of reflective modular forms}

\author{Haowu Wang}

\address{Center for Geometry and Physics, Institute for Basic Science (IBS), Pohang 37673, Korea}

\email{haowu.wangmath@gmail.com}


\date{\today}


\begin{abstract}
A modular form on an even lattice $M$ of signature $(l,2)$ is called reflective if it vanishes only on quadratic divisors orthogonal to roots of $M$. In this paper we show that every reflective modular form on a lattice of type $2U\oplus L$ induces a root system satisfying certain constrains. As applications, (1) we prove that there is no lattice of signature $(21,2)$ with a reflective modular form and that $2U\oplus D_{20}$ is the unique lattice of signature $(22,2)$ and type $U\oplus K$ which has a reflective Borcherds product; (2) we give an automorphic proof of Shvartsman and Vinberg's theorem, asserting that the algebra of modular forms for an arithmetic subgroup of $\mathrm{O}(l,2)$ is never freely generated when $l\geq 11$. We also prove several results on the finiteness of lattices with reflective modular forms.  
\end{abstract}

\maketitle

\section{Introduction}
The need to study modular forms on orthogonal groups $\Orth(l,2)$ was first pointed out by Weil \cite{Wei79} in his program for the study of $K3$ surfaces in the late 1950s. In 1988 Gritsenko \cite{Gri88} defined Jacobi forms of lattice index as Fourier--Jacobi coefficients of modular forms on $\Orth(l,2)$, generalizing Eichler and Zagier's monograph \cite{EZ85}. Gritsensko also introduced additive lifts to construct orthogonal modular forms in terms of Hecke operators of Jacobi forms, which is a generalization of the Saito--Kurokawa lift or Maass lift. In 1995 Borcherds \cite{Bor95} established a remarkable multiplicative lift to construct orthogonal modular forms with infinite product expansions, and noticed that the denominators of nice generalized Kac--Moody Lie algebras are orthogonal modular forms. Since then, the study of orthogonal modular forms has been active. 

Let $M$ be an even lattice of signature $(l,2)$ with a bilinear form $(-,-)$ and dual lattice $M'$, where $l\geq 3$.  The associated Hermitian symmetric domain $\cD(M)$ is defined as a connected component of 
\begin{equation*}
\{[\mathcal{Z}] \in  \PP(M\otimes \CC):  (\mathcal{Z}, \mathcal{Z})=0, (\mathcal{Z},\bar{\mathcal{Z}}) < 0\}.
\end{equation*}
Let $\Orth^+ (M)$ denote the orthogonal group preserving $\cD(M)$ and $M$. Fix a finite-index subgroup $\Gamma$ of $\Orth^+(M)$ and an integer $k$. A \textit{modular form} of weight $k$ and character $\chi$ for $\Gamma$ is a holomorphic function $F$ on the affine cone over $\cD(M)$ which satisfies
\begin{align*}
F(t\mathcal{Z})&=t^{-k}F(\mathcal{Z}), \quad \forall t \in \CC^\times,\\
F(g\mathcal{Z})&=\chi(g)F(\mathcal{Z}), \quad \forall g\in \Gamma.
\end{align*}

The Borcherds lift \cite{Bor95, Bor98} sends a weakly holomorphic modular form of weight $1-l/2$ for the Weil representation of $\SL_2(\ZZ)$ attached to $M'/M$ with integral principal part to a meromorphic modular form for the discriminant kernel
$$
\widetilde{\Orth}^+(M) = \{ g \in \Orth^+(M) : g(x)-x \in M, \quad \text{for all $x\in M'$} \}
$$
whose divisor is a linear combination of rational quadratic divisors
$$
v^\perp = \{ [\mathcal{Z}] \in \cD(M) : (\mathcal{Z}, v)=0 \}, \quad \text{for $v\in M$ with  $v^2>0$.}
$$
This form has an infinite product expansion at each $0$-dimensional cusp, so it is called a \textit{Borcherds product}. 
Borcherds constructed a holomorphic Borcherds product of weight $12$ for $\Orth^+(\II_{26,2})$ which vanishes precisely with multiplicity one on quadratic divisors $r^\perp$ for all $r\in \II_{26,2}$ with $r^2=2$. This form is denoted by $\Phi_{12}$ in the literature. Borcherds proved that $\Phi_{12}$ defines the denominator of the fake monster Lie algebra \cite{Bor90}. Motivated by Borcherds' form $\Phi_{12}$, Gritsenko and Nikulin \cite{GN98b} defined reflective modular forms and gave many examples. 

A non-constant modular form for $\Gamma < \Orth^+(M)$ is called \textit{reflective} if it vanishes only on quadratic divisors $r^\perp$ orthogonal to roots of $M$, where a \textit{root} $r\in M$ is a primitive positive-norm vector  such that the reflection 
$$
\sigma_r(x) = x-\frac{2(r,x)}{(r,r)}r, \quad x\in M
$$
fixes the lattice $M$, namely $\sigma_r \in \Orth^+(M)$. Reflective modular forms often have an infinite product expansion. 
If $F$ is a reflective modular form for some $\Gamma<\Orth^+(M)$, then the product of $F|g$ for all $g\in \Gamma' / (\Gamma\cap \Gamma')$ defines a reflective modular form for any other finite-index subgroup of $\Orth^+(M)$. By Bruinier's result \cite{Bru02, Bru14}, every reflective modular form for $\widetilde{\Orth}^+(M)$ equals a Borcherds product on $M$ (up to a nonzero constant factor) if $M$ splits as $U\oplus U(m)\oplus L$, where $U$ is the unique even unimodular lattice of signature $(1,1)$. If a reflective modular form for $\widetilde{\Orth}^+(M)$ is equal to a Borcherds product on $M$, then we call it a \textit{reflective Borcherds product}. 

Reflective modular forms have many significant applications. They are useful to classify and construct 
various algebraic objects, such as generalized Kac--Moody algebras \cite{Bor92, Bor98, GN98a, GN98b, GN02, Sch06}, hyperbolic reflection groups \cite{Bor00, GN98a} and free algebras of modular forms \cite{Wan21}. They also identify the geometric type of moduli spaces and modular varieties \cite{Bor96, GN96, GHS07, GH14, Gri18, Ma18}. 

Reflective modular forms are rare. In 1998 Gritsenko and Nikulin \cite[Conjecture 2.2.1]{GN98a} conjectured that the set of lattices with a reflective modular form is finite up to scaling, and posed the challenging problem of classifying these lattices. Over the past two decades, there have been many partial classification results of reflective modular forms  \cite{GN02, Bar03, Sch06, Sch17, Ma17, Ma18, Dit18, Wan18, GN18, Wan19, Wan22, Wan23a}, but this conjecture remains open. 

In this paper we prove some new classification results of reflective modular forms. In \cite{Wan19} the author established an effective approach for classifying $2$-reflective modular forms. This approach relies on the theory of Jacobi forms, and it yields that if $2U\oplus L$ has a $2$-reflective modular form then either $L$ has no $2$-roots or the sublattice of $L$ generated by $2$-roots has the full rank and satisfies some strong constrains. This argument was later generalized to classify reflective modular forms of lattices of prime level \cite{Wan22} and arithmetic subgroups $\Gamma < \Orth(l,2)$ for which the ring of modular forms is freely generated \cite{Wan21}.  In this paper we give  a final extension of this approach.

\begin{theorem}[see Theorem \ref{th:fake-root} for a full version]\label{MTH1}
Let $M=2U\oplus L$ and $\widetilde{\Orth}^+(M) < \Gamma < \Orth^+(M)$. Suppose that $F$ is a reflective modular form on $\Gamma$ whose zero divisor is a linear combination of quadratic divisors orthogonal to roots $r\in M$ satisfying $\sigma_r \in \Gamma$. Then $F$ induces a rescaled root system $\mathcal{R}$. Moreover, if $\mathcal{R}$ is nonempty then it has the same rank as $L$ and satisfies some constrains. In particular, $\mathcal{R}$ bounds the lattice $L$ from above and below.  
\end{theorem}

This result provides a possible way to classify reflective modular forms on lattices splitting $2U$ by classifying the induced root systems $\mathcal{R}$. As an application of Theorem \ref{MTH1}, we classify reflective modular forms on lattices of large rank.  

\begin{theorem}\label{MTH2}
\noindent
\begin{enumerate}
    \item There is no even lattice of signature $(l,2)$ with a reflective modular form when $l=21$ or $l\geq 23$ and $l\neq 26$.
    \item Let $M=U\oplus K$ be an even lattice of signature $(l,2)$ which has a reflective Borcherds product. Then $M\cong \II_{26,2}$ when $l=26$, and $M\cong 2U\oplus D_{20}$ when $l=22$. 
\end{enumerate}    
\end{theorem}

Theorem \ref{MTH2} has been proved in our previous work \cite{Wan18} for $l\neq 21$, $22$ by means of the differential operators on Jacobi forms. We will use Theorem \ref{MTH1} to prove the more subtle case $l=21$ or $22$. For any $l\leq 20$, there are indeed even lattices of signature $(l,2)$ with a reflective modular form. Statement (2) above does not hold if $M$ is not of type $U\oplus K$ (see Remarks \ref{rem:l=22} and \ref{rem:l=26}). 

In 1996 Esselmann \cite{Ess96} proved that (1) $U\oplus D_{20}$ is the unique reflective maximal even hyperbolic lattice of elliptic type and rank $22$; (2) there is no reflective even hyperbolic lattice of elliptic type and rank $l$ when $l=21$ or $l\geq 23$. Theorem \ref{MTH2} matches Esselmann's result by the connection between reflective modular forms and hyperbolic reflection groups. 

As another application of Theorem \ref{MTH1}, we give a new proof of Shvartsman and Vinberg's theorem, the original proof of which is in algebraic geometry. 

\begin{theorem}[\cite{SV17}]\label{MTH3}
The algebra of modular forms of integral weight and trivial character for an arithmetic subgroup of $\Orth(l,2)$ is never freely generated when $l>10$.   
\end{theorem}

The new proof is based on a previous result of the author. It was proved in \cite{Wan21} that if the algebra of modular forms for $\Gamma$ is freely generated then the modular Jacobian of generators is a reflective modular form on $\Gamma$ with simple zeros. Applying Theorem \ref{MTH1} to the Jacobian leads to a contradiction when $l\geq 11$, thus proving Theorem \ref{MTH3}.   

In this paper we also prove the finiteness of lattices with reflective modular forms. The previous best result in this direction was attributed to Ma \cite[Corollary 1.10]{Ma18}, proving the finiteness of lattices which has a reflective modular form with simple zeros. 

\begin{theorem}\label{MTH4}
\noindent
\begin{enumerate}
    \item The set of even lattices of type $2U\oplus L$ with a reflective modular form is finite. 
    \item When $l\geq 7$, the set of even lattices of signature $(l,2)$ and type $U\oplus K$ which has a reflective Borcherds product vanishing on $r^\perp$ is finite, where $r\in U$ and $r^2=2$. 
    \item When $l\geq 14$, every reflective Borcherds product on an even lattice of type $U\oplus K$ and signature $(l,2)$ vanishes on $r^\perp$ for $r\in U$ with $r^2=2$. 
\end{enumerate}    
\end{theorem}

It was proved by Borcherds \cite{Bor98}, Gritsenko--Nikulin \cite{GN98a}  and Looijenga \cite{Loo03} that a lattice with a reflective modular form may induce a reflective hyperbolic lattice.  We conclude the proof of  Theorem \ref{MTH4} from this result and the finiteness of reflective hyperbolic lattices, which was proved by Nikulin \cite{Nik80b, Nik96}. 

This paper is organized as follows. In Section \ref{sec:root-systems} we present a full version of Theorem \ref{MTH1} and give a proof. Section \ref{sec:classification} is devoted to the proof of Theorem \ref{MTH2}. In Section \ref{sec:non-free} we prove Theorem \ref{MTH3}. In Section \ref{sec:finiteness} we introduce reflective hyperbolic lattices and complete the proof of Theorem \ref{MTH4}.

\section{The full version of Theorem \ref{MTH1} and its proof}\label{sec:root-systems}
\subsection{The input of reflective Borcherds products}
Let $M$ be an even lattice of signature $(l,2)$ with $l\geq 3$. The input $f$ of a Borcherds product $F$ on $M$ is a weakly holomorphic modular form of weight $1-l/2$ for the Weil representation $\rho_M$. More precisely, $\rho_M$ is a unitary representation of a double cover of $\SL_2(\ZZ)$ on the group ring $\CC[M'/M]=\mathrm{span}(e_\gamma: \gamma \in M'/M)$ (see e.g. \cite{Bor98, Bru02} for details). The input $f$ is represented by a Fourier series of the form
$$
f(\tau) = \sum_{\gamma\in M'/M} \sum_{n\in \ZZ-\gamma^2/2}c(\gamma,n)q^n e_\gamma, \quad \tau \in \HH, \; q=e^{2\pi i \tau}. 
$$
Note that $c(\gamma,n)\in \ZZ$ if $n<0$ and there are only finitely many nonzero Fourier coefficients $c(\gamma,n) q^n e_\gamma$ with $n<0$.  Their sum is called the \textit{principal} part of $f$, which determines the zeros and poles of $F$. Let $\lambda$ be a primitive positive-norm vector of $M'$. The quadratic divisor $\lambda^\perp$ has multiplicity $\sum_{d=1}^\infty c(d\lambda, -d^2\lambda^2/2)$ in the divisor of $F$. The weight of $F$ is given by $c(0,0)/2$. 

We now describe the input of a reflective Borcherds product. Let $\lambda$ be a primitive positive-norm vector of $M'$. The quadratic divisor $\lambda^\perp$ is reflective, i.e. $\sigma_\lambda \in \Orth^+(M)$ if and only if there exists a positive integer $d$ such that $\lambda^2=2/d$ and $d\lambda \in M$. In fact, when $\lambda^\perp$ is reflective, the order of $\lambda$ in $M'/M$ is either $d$, or $d/2$ in which case $d/2$ is even.

\begin{lemma}\label{lem:principal}
Let $M=U\oplus K$ be an even lattice of signature $(l,2)$ with $l\geq 3$. Then the principal part of the input of a reflective Borcherds product on $M$ has the form
$$
\sum_{t=1}^\infty \sum_{\substack{x\in K'/K\\ \ord(x)=t}} c(x, -1/t)q^{-1/t}e_x + \sum_{t=1}^\infty \sum_{\substack{y\in K'/K\\ \ord(y)=2t}} c(y, -1/(4t))q^{-1/(4t)}e_y,
$$
where $c(x,-1/t)\geq 0$, $c(2y,-1/t) + c(y, -1/(4t)) \geq 0$ and $\ord(\gamma)$ denotes the order of $\gamma$ in $K'/K$. 
\end{lemma}
\begin{proof}
Let $F$ denote the reflective Borcherds product and $f$ denote its input. We write a vector $\lambda\in U\oplus K'$ as $(a,v,b)$ with norm $\lambda^2=v^2-2ab$ for $a,b\in \ZZ$ and $v\in K'$. We expand $f$ as 
$$
f(\tau)=\sum_{\gamma\in K'/K}\sum_{n\in \ZZ - \gamma^2/2} c(\gamma,n) q^n e_\gamma.
$$
For $c(\gamma,n)\neq 0$ and $n<0$, let $d$ denote the largest integer such that $c(d\gamma,d^2n)\neq 0$. We view $\gamma$ as a vector in $K'$. Then $(d^2n+d^2\gamma^2/2, d\gamma, 1)^\perp$ has multiplicity $c(d\gamma,d^2n)$ in the divisor of $F$. Thus $c(d\gamma,d^2n)>0$ and there exists a positive integer $m$ such that $-d^2 n = 1/m$ and $md\gamma \in K$ because the divisor is reflective. We deduce from $(\gamma,md\gamma)\in \ZZ$ that $d=1$ or $2$. We see from $(\gamma,\ord(\gamma)\gamma)\in \ZZ$ that $d^2m | 2\cdot \ord(\gamma)$. Thus $\ord(\gamma)=m$ or $m/2$ if $d=1$, and $\ord(x)=2m$ if $d=2$. 

When $d=1$ and $\ord(\gamma)=m$, $\gamma$ is of type $x$ with $t=m$. 
When $d=1$ and $\ord(\gamma)=m/2$, we derive from $(m\gamma/2)^2 \in 2\ZZ$ that $m/2$ is even and thus $\gamma$ is of type $y$ with $t=m/4$. 
When $d=2$, $\gamma$ is of type $y$ with $t=m$. We have $c(\gamma,n)+c(2\gamma,4n)\geq 0$ because this integer is the multiplicity of $(n+\gamma^2/2,\gamma,1)^\perp$ in the divisor of $F$. Therefore, the principal part of $f$ has the desired form.     
\end{proof}

Let $M=2U\oplus L$, where $L$ is an even positive definite lattice of rank $\rank(L)$. By Bruinier's result \cite{Bru02}, every reflective modular form for $\widetilde{\Orth}^+(M)$ is essentially a Borcherds product. A Borcherds product $F$ on $M$ has a representation in terms of Jacobi forms, because its input $f(\tau)$ can be realized as a weakly holomorphic Jacobi form of weight $0$ and lattice index $L$ (see \cite[Theorem 4.2]{Gri18} or \cite[Section 4]{Wan19}). The Jacobi form input of $F$ can be expanded into Fourier series 
$$
\phi(\tau, \mathfrak{z}) = \sum_{n\in \ZZ,\; \ell \in L'} f(n,\ell)q^n \zeta^\ell, \quad \mathfrak{z}\in L\otimes\CC, \; \zeta^\ell = e^{2\pi i (\ell,\mathfrak{z})}. 
$$
The Fourier coefficient $c(x,m)q^m e_x$ of $f(\tau)$ corresponds to the Fourier coefficient $f(n,\ell)q^n\zeta^\ell$ of $\phi(\tau,\mathfrak{z})$, where $\ell$ is an arbitrary vector in $x + L$ and $n=m+\ell^2/2$. In particular, $c(x,m)=f(n,\ell)$ and $f(n,\ell)$ is integral if $2n<(\ell,\ell)$.

Suppose that the above $F$ is reflective.
We derive from Lemma \ref{lem:principal} that the Fourier expansion of $\phi$ satisfies
\begin{equation}\label{eq:reflective-Jacobi}
\phi(\tau, \mathfrak{z}) = f(-1,0)q^{-1} + \sum_{t=1}^\infty \sum_{\substack{r\in L'\\ r^2=2/t \\ \ord(r)=t}} f(0,r) \zeta^r + \sum_{t=1}^\infty \sum_{\substack{s\in L'\\ s^2=1/(2t) \\ \ord(s)=2t}} f(0,s) \zeta^s + f(0,0) + O(q).
\end{equation}
In the above expansion, $f(0,0)/2$ gives the weight of $F$, $f(-1,0)$ and $f(0,r)$ are non-negative integers, $f(0,s)$ are integers, and $f(0,s)+f(0,2s)\geq 0$. Furthermore, $f(0,r)=f(-1,0)$ if $r^2=2$. 

\subsection{Root systems} We review some basics of root systems and define fake root systems. 

Let $R$ be a usual irreducible root system of rank $\rank(R)$ with the normalized bilinear form $\latt{-,-}$ such that $\latt{\alpha,\alpha}=2$ for long roots $\alpha$ (When all roots of $R$ have the same length, we call them long roots.).  We denote by $\rho_l$ (resp. $\rho_s$) half the sum of all positive long (resp. short) roots of $R$. Then $\rho=\rho_l+\rho_s$ is the Weyl vector of $R$. We define
$$
h_l = \frac{1}{\rank(R)}\sum_\alpha \latt{\alpha,\alpha}, \quad h_s = \frac{1}{\rank(R)}\sum_\beta \latt{\beta,\beta},
$$
where $\alpha$ takes over all positive long roots and $\beta$ takes over all positive short roots. We have
$$
\sum_\alpha \latt{\alpha, \mathfrak{z}}^2 = h_l\latt{\mathfrak{z},\mathfrak{z}}, \quad \sum_\beta \latt{\beta, \mathfrak{z}}^2 = h_s\latt{\mathfrak{z},\mathfrak{z}}, \quad \mathfrak{z}\in R\otimes\CC.
$$
The number $h=h_l+h_s$ is called the dual Coxeter number of $R$. Let $|R_l|$ and $|R_s|$ denote the number of long roots and short roots, respectively. We set $|R|=|R_l|+|R_s|$. 

Let $d$ be a positive integer and $R(1/d)$ denote the rescaled root system with bilinear form $(-,-)=\latt{-,-}/d$. 
In the following, for each $R(1/d)$ we introduce an associated fake root system $\hat{R}$ with some invariants. The positive integers $a$ and $b$ are called the multiplicities of long roots $\alpha$ and short roots $\beta$, respectively.  The integer $c$ satisfying $a+c\geq 0$ is called the multiplicity of the fake root $\alpha/2$. 
We call $\hat{\rho}$ the fake Weyl vector, $\hat{h}$ the fake Coxeter number, and $|\hat{R}|$ the number of all fake roots (counting multiplicities). The motivation to define these notions can be found in the proof of Theorem \ref{th:fake-root}.

\begin{enumerate}
    \item $(A_n,a;d)$ for $n\geq 2$. $h=n+1$ and $|R|=nh$.
    $$
    \hat{\rho}:=a\rho, \quad  |\hat{R}|:=a|R|, \quad \hat{h}:=\frac{ah}{d}. 
    $$
    \item $(D_n,a;d)$ for $n\geq 4$. $h=2(n-1)$ and $|R|=nh$.
    $$
    \hat{\rho}:=a\rho, \quad  |\hat{R}|:=a|R|, \quad \hat{h}:=\frac{ah}{d}. 
    $$
    \item $(E_n,a;d)$ for $n=6,7,8$. The corresponding $h=12, 18, 30$ and $|R|=nh$. 
    $$
    \hat{\rho}:=a\rho, \quad  |\hat{R}|:=a|R|, \quad \hat{h}:=\frac{ah}{d}. 
    $$
    \item $(B_n,a,b;d)$ for $n\geq 3$. $h_l=2(n-1)$, $h_s=1$, $|R_l|=2n(n-1)$ and $|R_s|=2n$. 
    $$
    \hat{\rho}:=a\rho_l+b\rho_s, \quad  |\hat{R}|:=a|R_l|+b|R_s|, \quad \hat{h}:=\frac{ah_l+bh_s}{d}. 
    $$
    \item $(G_2,a,b;d)$. $h_l=3$, $h_s=1$, $|R_l|=6$ and $|R_s|=6$. 
    $$
    \hat{\rho}:=a\rho_l+b\rho_s, \quad  |\hat{R}|:=a|R_l|+b|R_s|, \quad \hat{h}:=\frac{ah_l+bh_s}{d}. 
    $$
    \item $(F_4,a,b;d)$. $h_l=6$, $h_s=3$, $|R_l|=24$ and $|R_s|=24$. 
    $$
    \hat{\rho}:=a\rho_l+b\rho_s, \quad  |\hat{R}|:=a|R_l|+b|R_s|, \quad \hat{h}:=\frac{ah_l+bh_s}{d}. 
    $$
    \item $(A_1,a|c;d)$ for $a+c\geq 0$. $h=2$ and $|R|=2$. 
    $$
    \hat{\rho}:=(a+c/2)\rho, \quad  |\hat{R}|:=(a+c)|R|, \quad \hat{h}:=\frac{ah}{d}+\frac{ch}{4d}. 
    $$
    \item $(C_n,a|c,b;d)$ for $n\geq 2$ and $a+c\geq 0$. $h_l=2$, $h_s=n-1$, $|R_l|=2n$ and $|R_s|=2n(n-1)$. 
    $$
    \hat{\rho}:=(a+c/2)\rho_l+b\rho_s, \quad  |\hat{R}|:=(a+c)|R_l|+b|R_s|, \quad \hat{h}:=\frac{ah_l+bh_s}{d}+\frac{ch_l}{4d}. 
    $$
\end{enumerate}

It is known that the denominator of the affine Lie algebra of type $A_1$ is given by the odd Jacobi theta function (see \cite{GSZ19})
$$
\vartheta(\tau,z)=q^{\frac{1}{8}}(\zeta^{\frac{1}{2}} - \zeta^{-\frac{1}{2}})\prod_{n=1}^\infty(1-q^n\zeta)(1-q^n\zeta^{-1})(1-q^n), \quad z\in \CC, \; \zeta= e^{2\pi iz}. 
$$
Using the Jacobi theta function and the Dedekind eta function  
$$
\eta(\tau)=q^{\frac{1}{24}}\prod_{n=1}^\infty (1-q^n),
$$
we define a weight-zero theta block associated with $\hat{R}$ as
\begin{equation}\label{eq:fake-denominator}
\vartheta_{\hat{R}}(\tau,\mathfrak{z})= \prod_{\alpha} \left(\frac{\vartheta(\tau, \latt{\alpha, \mathfrak{z}})}{\eta(\tau)} \right)^a \prod_{\beta} \left(\frac{\vartheta(\tau, \latt{\beta, \mathfrak{z}})}{\eta(\tau)} \right)^b \prod_{\alpha} \left(\frac{\vartheta(\tau, \latt{\alpha/2, \mathfrak{z}})}{\eta(\tau)} \right)^c,
\end{equation}
where $\mathfrak{z}\in R\otimes\CC$, $\alpha$ runs over all positive long roots and $\beta$ runs over all positive short roots. When $a=b=1$ and $c=0$, $\hat{R}$ reduces to the rescaled root system $R(1/d)$, and $\eta(\tau)^{\rank(R)}\vartheta_{\hat{R}}(\tau,\mathfrak{z})$ is identical to the product side of the denominator identity of the affine Lie algebra of type $R$. 

\subsection{The main result} We now state the full version of Theorem \ref{MTH1}.

\begin{theorem}\label{th:fake-root}
Let $M=2U\oplus L$ and $\widetilde{\Orth}^+(M) < \Gamma < \Orth^+(M)$. Let 
$F$ be a reflective Borcherds product of weight $k$ on $M$. Suppose that if $F$ vanishes on $r^\perp$ then $\sigma_r \in \Gamma$.  Let $\phi$ denote the Jacobi form input of $F$ with Fourier expansion \eqref{eq:reflective-Jacobi}. We introduce two sets
\begin{align*}
\mathcal{R} &= \{ r : f(0,r) > 0 \} \cup \{ s : f(0,s)>0, \; f(0,2s)=0 \},\\
\widehat{\mathcal{R}} &= \{ (\ell, f(0,\ell)) : 0\neq \ell \in L', \; f(0,\ell) \neq 0\}.
\end{align*}

If $\mathcal{R}$ is empty, then $k=12f(-1,0)$, the Weyl vector of $F$ has the form $(f(-1,0),0,0)$, and the first nonzero Fourier--Jacobi coefficient of $F$ is $\eta(\tau)^{2k}$. 

Assume that $\mathcal{R}$ is nonempty. Then $\mathcal{R}$ is a rescaled root system of rank $\rank(L)$. We decompose $\mathcal{R}$ into irreducible components as
$$
\mathcal{R} = R_1(1/d_1)\oplus \cdots \oplus R_m(1/d_m), 
$$
where $d_j$ are positive integers. Let $\hat{R}_j$ be the fake root system associated with $R_j(1/d_j)$ for which the multiplicity of each  root $\ell$ is given by $f(0,\ell)$. Then we have the isomorphism
$$
\widehat{\mathcal{R}} \cong \hat{R}_1 \oplus \cdots \oplus \hat{R}_m.
$$
\begin{enumerate}
    \item The invariants of $\hat{R}_j$ satisfy the following identities:
    \begin{equation}\label{eqA}
    C:=\frac{2k+\sum_{j=1}^m |\hat{R}_j|}{24} - f(-1,0) = \hat{h}_j, \quad 1\leq j \leq m. 
    \end{equation}
    When $f(-1,0)=0$, $d_j\geq 2$ for every $j$. When $f(-1,0)>0$, for any $\hat{R}_j$ with $d_j=1$, the multiplicity of long roots is $a_j=1$. 
    \item The Weyl vector of $F$ has the form $\rho_F =( C+f(-1,0), \; \sum_{j=1}^m \hat{\rho}_j,\; C )$. In particular, $\rho_F\neq 0$ and it has non-positive norm, i.e.
    \begin{equation}\label{eqB}
     \sum_{j=1}^m \frac{\latt{\hat{\rho}_j, \hat{\rho}_j}}{d_j} - 2C\Big(C+f(-1,0)\Big) \leq 0.  
    \end{equation}
    \item The rescaled lattice $L(C)$ is integral. The first nonzero Fourier--Jacobi coefficient of $F$ is 
    \begin{equation}\label{eqC}
     \vartheta_F(\tau,\mathfrak{z}) =  \eta(\tau)^{2k} \bigotimes_{j=1}^m \vartheta_{\hat{R}_j}(\tau, \mathfrak{z}_j), \quad \mathfrak{z}=(\mathfrak{z}_j)_{j=1}^m \in L\otimes\CC.    
    \end{equation}
    The function $\vartheta_F(\tau,\mathfrak{z})$ is a holomorphic Jacobi form of weight $k$ and lattice index $L(C)$. 
    \item Let $Q_j$ be the even lattice generated by long roots of $R_j$ and $P_j$ be the rational lattice
    $$
    \{ x \in R_j\otimes\QQ : \latt{x, v}\in \ZZ, \; v\in R_j \}.
    $$
    We denote by $J$ by the set of $j$ such that $R_j$ is of type $E_8$, $F_4$ or $G_2$, or $\hat{R}_j$ is of type $(A_1,a|c;d_j)$ or $(C_n,a|c,b;d_j)$ with $c\neq 0$. Then we have the bound
    $$
    L=T\oplus \bigoplus_{j\in J}Q_j(d_j), \quad 
    \bigoplus_{j\not\in J} Q_j(d_j) < T < \bigoplus_{j\not\in J} P_j(d_j). 
    $$
\end{enumerate}
\end{theorem}

\begin{proof}
Applying \cite[Proposition 2.6]{Gri18} to the Jacobi form input $\phi$ we obtain
\begin{align}
C:=\frac{1}{24}\sum_{\ell\in L'} f(0,\ell)  - f(-1,0) &= \frac{1}{2\rank(L)}\sum_{\ell \in L'} f(0,\ell)(\ell,\ell), \label{eq1}\\   
\sum_{\ell\in L'} f(0,\ell)(\ell,\mathfrak{z})^2 &= 2C(\mathfrak{z},\mathfrak{z}) \label{eq2}.
\end{align}
By \cite[Theorem 4.2]{Gri18}, the Weyl vector of $F$ has the form $(C+f(-1,0),\; \frac{1}{2}\sum_{\ell>0}f(0,\ell)\ell,\; C)$, and the first nonzero Fourier--Jacobi coefficient of $F$ is given by the theta block
\begin{equation}\label{eq3}
\vartheta_F(\tau,\mathfrak{z}) = \eta(\tau)^{2k}\prod_{\ell > 0} \left( \frac{\vartheta(\tau,(\ell,\mathfrak{z}))}{\eta(\tau)}\right)^{f(0,\ell)}. \end{equation}

If $\mathcal{R}$ is empty, i.e. $f(0,\ell)=0$ for all $\ell\neq 0$, then $C=0$ and $k=f(0,0)/2=12f(-1,0)$ by \eqref{eq1}. Furthermore, the Weyl vector and the leading Fourier--Jacobi coefficient have the desired form.

Assume that $\mathcal{R}$ is nonempty. We see from the Fourier expansion \eqref{eq:reflective-Jacobi} and
$$
C=\frac{1}{2\rank(L)} \left( \sum_{s} \big(4f(0,2s)+f(0,s)\big)(s,s) + \sum_{r\neq 2s} f(0,r)(r,r) \right)
$$
that $C>0$. Therefore, $\mathcal{R}$ generates $L\otimes\QQ$ over $\QQ$ by \eqref{eq2}.

We now prove that $\mathcal{R}$ is a rescaled root system. If $v \in \mathcal{R}$ then $-v \in \mathcal{R}$, because $\phi(\tau,\mathfrak{z})=\phi(\tau,-\mathfrak{z})$ and thus $f(n,\ell)=f(n,-\ell)$ for all $n\in\ZZ$ and $\ell\in L'$. The shape of \eqref{eq:reflective-Jacobi} yields that if $v\in \mathcal{R}$ then $mv\not\in\mathcal{R}$ for any $m>1$. We write a vector $\lambda \in 2U\oplus L'$ as $(n_1,n_2,\ell,n_3,n_4)$ for $n_1,n_2,n_3,n_4\in\ZZ$ and $\ell \in L'$ with $\lambda^2=\ell^2-2(n_1n_4+n_2n_3)$.  Let $v, u\in \mathcal{R}$. We define $\lambda_v=(0,0,v,1,0)$ and $\lambda_u=(0,0,u,1,0)$. Then $F$ vanishes on $\lambda_u^\perp$ with multiplicity $f(0,u)$. By assumption, $\sigma_{\lambda_v}$ lies in $\Gamma$ and fixes $M=2U\oplus L$. Thus $F$ also vanishes on the quadratic divisor orthogonal to $\sigma_{\lambda_v}(\lambda_u)$. From 
$$
\sigma_{\lambda_v}(\lambda_u) = (0,0,\sigma_v(u),1-2(v,u)/(v,v),0)
$$
we derive that $\sigma_{v}(u)\in L'$ and $2(v,u)/(v,v)\in \ZZ$. Notice that both $\lambda_{v,u}:=(0,0,\sigma_v(u),1,0)$ and $\sigma_{\lambda_v}(\lambda_u)$ are primitive in $M$. By the Eichler criterion (see e.g. \cite[Proposition 3.3]{GHS09}), there exists $g\in \widetilde{\Orth}^+(M)$ such that $g(\sigma_{\lambda_v}(\lambda_u))=\lambda_{v,u}$. It follows that $F$ vanishes on $\lambda_{v,u}^\perp$ with the same multiplicity as $\lambda_u^\perp$. Note that $(u,u)=(\sigma_v(u),\sigma_v(u))$ and $\ord(u)=\ord(\sigma_v(u))$. If $u$ is of type $r$, then $\sigma_v(u)$ is also of type $r$ and $\lambda_{v,u}^\perp$ has multiplicity $f(0,\sigma_v(u))>0$. If $u$ is of type $s$, then $\sigma_v(u)$ is also of type $s$. In this case, $F$ does not vanish on $(0,0,2u,1,0)^\perp$ and thus $F$ does not vanish on $(0,0,2\sigma_v(u),1,0)^\perp$. Therefore, $f(0,2\sigma_v(u))=0$ and $F$ vanishes on $\lambda_{v,u}^\perp$ with multiplicity $f(0,\sigma_v(u))>0$. It follows that $\sigma_v(u)\in \mathcal{R}$. We then prove that  $\mathcal{R}$ is a rescaled root system. Since $\mathcal{R}$ generates $L\otimes\QQ$, it has the same rank as $L$. 

We decompose $\mathcal{R}$ into irreducible components as in the theorem. By \eqref{eq:reflective-Jacobi}, for any nonzero $\ell \in L'$ with $f(0,\ell)\neq 0$, $\ell^2=2/t$ for some positive integer. Thus the scales $d_j$ are positive integers. We explain the invariants of fake root systems by cases.
\begin{enumerate}
    \item $R_j$ is of type $A_n$ for $n\geq 2$, $D_n$ for $n\geq 4$, $E_6$, $E_7$ or $E_8$.  Let $\alpha$ be a root of $R_j$. Since $R_j(1/d_j) \subset L'$, we can view $\alpha$ as a vector in $L'$ of norm $\alpha^2=2/d_j$ and $d_j\alpha \in L$. By the type of $R_j$, there exists a root $\alpha'$ such that $\latt{\alpha,\alpha'}=1$. Thus $\ord(\alpha)=d_j$ in $L'/L$ and $\alpha$ is primitive in $L'$. In particular, $\alpha \in L'$ is a vector of type $r$ and $\alpha/2 \not\in L'$. Note that $a=f(0,\alpha)$. 
    \item $R_j$ is of type $B_n$ for $n\geq 3$. Let $\alpha \in R_j$ be a long root and $\beta\in R_j$ be a short root. There exist a long root $\alpha'$ and a short root $\beta' \in R_j$ such that $\latt{\alpha, \alpha'}=\latt{\beta,\beta'}=1$. Similarly to (1), we show that $\alpha$ is a primitive vector in $L'$ with $\alpha^2=2/d_j$ and $\ord(\alpha)=d_j$. Thus $\alpha\in L'$ is a vector of type $r$ and $\alpha/2\not\in L'$. Note that a long root can be written as a sum of two distinct short roots. Under $R_j(1/d_j) \subset L'$, $\beta$ is a primitive vector in $L'$ with $\beta^2=1/d_j$ and $2d_j\beta\in L$.  There are two cases:
    \begin{enumerate}
        \item[(i)] $\ord(\beta)=2d_j$. In this case, $\beta\in L'$ is of type $r$ and $\beta/2 \not\in L'$. 
        \item[(ii)] $\ord(\beta)=d_j$. In this case, $d_j$ is even and $\beta\in L'$ is of type $s$. 
    \end{enumerate}
    Note that $a=f(0,\alpha)$ and $b=f(0,\beta)$. 
    \item $R_j$ is of type $F_4$. Similarly,  a long root $\alpha$ is a primitive vector of $L'$ with $\alpha^2=2/d_j$ and $\ord(\alpha)=2/d_j$, so $\alpha\in L'$ is of type $r$ and $\alpha/2\not\in L'$. A short root $\beta$ is a primitive vector of $L'$ with $\beta^2=1/d_j$ and $2d_j\beta\in L$. There exists a short root $\beta'$ such that $\latt{\beta,\beta'}=1/2$. Thus $\ord(\beta)=2d_j$. Therefore, $\beta\in L'$ is of type $r$ and $\beta/2 \not\in L'$.  Note that $a=f(0,\alpha)$ and $b=f(0,\beta)$. 
    \item $R_j$ is of type $G_2$. A long root $\alpha$ is a primitive vector of $L'$ with $\alpha^2=2/d_j$ and $\ord(\alpha)=2/d_j$, so $\alpha\in L'$ is of type $r$ and $\alpha/2\not\in L'$. A short root $\beta$ is a primitive vector of $L'$ with $\beta^2=2/3d_j$ and $3d_j\beta\in L$. There exists a short root $\beta'$ such that $\latt{\beta,\beta'}=1/3$. Thus $\ord(\beta)=3d_j$. Therefore, $\beta\in L'$ is of type $r$ and $\beta/2 \not\in L'$. Note that $a=f(0,\alpha)$ and $b=f(0,\beta)$. 
    \item $R_j$ is of type $A_1$. The unique positive root $\alpha$ is a vector of $L'$ with $\alpha^2=2/d_j$ and $d_j\alpha \in L$. Then $a=f(0,\alpha)$. There are two cases:
    \begin{enumerate}
        \item[(i)] $\ord(\alpha)=d_j$. In this case, $\alpha\in L'$ is of type $r$. When $\alpha/2\in L'$, the parameter $c$ introduced before occurs and $c=f(0,\alpha/2)$. When $\alpha/2\not\in L'$, $c=0$.
        \item[(ii)] $\ord(\alpha)=d_j/2$. In this case, $d_j/2$ is even, $\alpha\in L'$ is primitive and of type $s$, and $c=0$ because $\alpha/2\not\in L'$.
    \end{enumerate}
    \item $R_j$ is of type $C_n$ for $n\geq 2$. A short root $\beta$ is a vector of $L'$ with $\beta^2=1/d_j$ and $2d_j \beta \in L$. These exists another short root $\beta'$ such that $\latt{\beta,\beta'}=1/2$. Thus $\ord(\beta)=2d_j$ and $\beta$ is primitive in $L'$. Therefore, $\beta\in L'$ is of type $r$ and $\beta/2\not\in L'$. A long root $\alpha$ is a vector of $L'$ with $\alpha^2=2/d_j$ and $d_j\alpha\in L$. Note that $a=f(0,\alpha)$ and $b=f(0,\beta)$. There are two cases:
    \begin{enumerate}
        \item[(i)] $\ord(\alpha)=d_j$. In this case, $\alpha\in L'$ is of type $r$. When $\alpha/2\in L'$, the parameter $c$ introduced before occurs and $c=f(0,\alpha/2)$. When $\alpha/2\not\in L'$, $c=0$. 
        \item[(ii)] $\ord(\alpha)=d_j/2$. In this case, $d_j/2$ is even, $\alpha\in L'$ is primitive and of type $s$, and $c=0$ because $\alpha/2\not\in L'$.
    \end{enumerate}
\end{enumerate}

From the above discussions we see that $c\neq 0$ corresponds to vectors of type $s$ with $f(0,s)\neq 0$ and $f(0,2s)>0$. To prove Statement (1), we apply \eqref{eq2} to each $\hat{R}_j$ and apply \eqref{eq1} to $\hat{\mathcal{R}}$, and use the basic fact: in the divisor of $F$, $(0,-1,0,1,0)^\perp$ has multiplicity $f(-1,0)$ and $(0,0,v,1,0)^\perp$ for $v\in L$ with $v^2=2$ has multiplicity $f(0,v)$, and the set of these divisors is transitive under $\widetilde{\Orth}^+(M)$. 

Statement (2) follows from the general form of the Weyl vector and Statement (3) follows from \cite[Corollary 4.5]{Wan19}, \eqref{eq3} and \eqref{eq:fake-denominator}. 

We now prove the last statement. By construction, the lattice generated by $\mathcal{R}$ over $\ZZ$ is contained in $L'$, so its dual contains $L$. It implies that $L<\oplus_j P_j(d_j)$. For any $v \in \mathcal{R}$, if $v^2=2/d$ as a vector of $L'$, then $dv \in L$. Thus $L$ contains the lattice generated by these vector $dv$ over $\ZZ$, which is exactly $\oplus_j Q_j(d_j)$.  When $R_j$ is of type $E_8$, $F_4$ or $G_2$, we have $P_j=Q_j$. When $\hat{R}_j$ is of type $(A_1,a|c;d_j)$ or $(C_n,a|c,b;d_j)$ with $c\neq 0$, the lattice generated by half long roots $\alpha/2$ and short roots $\beta$ is contained in $L'$, so its dual contains $L$, which is actually $Q_j(d_j)$. We then prove the desired bound.
\end{proof}

We formulate lattices of types $P$ and $Q$ defined in Theorem \ref{th:fake-root} (4) in Table \ref{tab:data} for convenience.  

\begin{table}[ht]
\caption{Lattices $P$ and $Q$ associated with a root system $R$}
\label{tab:data}
\renewcommand\arraystretch{1.5}
\[
\begin{array}{|c|c|c|c|c|c|c|c|c|c|}
\hline 
R & A_n & B_n & C_n & D_n & E_6 & E_7 & E_8 & G_2 & F_4  \\ 
\hline 
Q & A_n & D_n & nA_1 & D_n & E_6 & E_7 & E_8 & A_2 & D_4 \\
\hline
P & A_n' & \ZZ^n & D_n'(2) & D_n' & E_6' & E_7' & E_8 & A_2 & D_4 \\
\hline
\end{array} 
\]
\end{table}

\begin{remark}
There are indeed reflective Borcherds products which do not satisfy the assumption in Theorem \ref{th:fake-root}. For example, there is a reflective Borcherds product $F_{24}$ of weight $24$ with simple zeros for $\Orth^+(2U\oplus D_4)$ (see \cite[Theorem 4.4]{Woi17}). The fake root system of $F_{24}$ is $(D_4,1;2)$ and the associated identity of type \eqref{eqA} is  $(24+2\times 24)/24 = 3$. We can decompose $F_{24}$ as a product of three reflective Borcherds products of weight $8$.  One of them has fake root system $(A_1,1|0;2)^4$. Let $F$ be one of the other two products. Then $F$ vanishes on some $\lambda^\perp$ such that $F$ is not modular under the associated reflection $\sigma_\lambda$.    
\end{remark} 

\begin{example}\label{example}
We give many examples of fake root systems of reflective Borcherds products.
\begin{enumerate}
    \item The Igusa cusp forms of weights $10$ and $35$ for $\Sp_4(\ZZ)$ can be realized as reflective Borcherds products for $\Orth^+(2U\oplus A_1)$ (see \cite{GN98b}). The fake root systems of the two forms and their product are respectively $(A_1,2|0;2)$, $(A_1,1|0;1)$ and $(A_1,1|2;1)$. 
    \item The form $D_{1/2}$ in \cite[Theorem 1.11]{GN98b} is a reflective Borcherds product of weight $1/2$ with simple zeros for $\Orth^+(2U\oplus A_1(36))$. Its associated fake root system is $(A_1,1|-1;36)$.
    \item Gritsenko \cite[Section 6.6]{Gri18} constructed a reflective Borcherds product of weight $42$ for $\Orth^+(2U\oplus 2E_8\oplus 2A_1)$, whose associated fake root system is $(E_8,1;1)^2\oplus (C_2,1|32,12;1)$. 
    \item Gritsenko \cite[Section 6.6]{Gri18} constructed a reflective Borcherds product of weight $48$ for $\Orth^+(2U\oplus 2E_8\oplus A_2)$, whose associated fake root system is $(E_8,1;1)^2\oplus (G_2,1,27;1)$. 
    \item There is a reflective Borcherds product of weight $75$ for $\Orth^+(2U\oplus 2E_8\oplus A_1)$ (see e.g. \cite[Section 6.6]{Gri18}), whose associated fake root system is $(E_8,1;1)^2\oplus (A_1,1|56;1)$. 
    \item Gritsenko \cite[Section 6.6]{Gri18} constructed a reflective Borcherds product of weight $54$ for $\Orth^+(2U\oplus 2E_8\oplus A_1(2))$, whose associated fake root system is $(E_8,1;1)^2\oplus (A_1,14|64;2)$. 
    \item Borcherds \cite{Bor00} constructed a reflective modular form $\Psi_{24}$ of weight $24$ on the lattice of genus $\II_{22,2}(2_{\II}^{-2})$. We consider lattice models
    $$
    2U\oplus D_{20} \cong 2U\oplus E_8\oplus D_{12} \cong 2U\oplus 2E_8\oplus D_4.
    $$
    The corresponding fake root systems are respectively 
    $$
    (B_{20}, 1,8;1), \quad (E_8,1;1)\oplus (B_{12},1,8;1), \quad (E_8,1;1)^2\oplus (F_4,1,8;1).
    $$
\end{enumerate}    
\end{example}

\section{A proof of Theorem \ref{MTH2}}\label{sec:classification}
In this section we use Theorem \ref{th:fake-root} to prove Theorem \ref{MTH2}. We first prove some useful lemmas. 

\begin{lemma}\label{lem:sym-bound}
Let $M=U\oplus K$ be an even lattice of signature $(l,2)$ with $l\geq 3$. If $M$ has a reflective Borcherds product for which the coefficient of $q^{-1}e_0$ in the input is zero, then $l\leq 13$.  
\end{lemma}
\begin{proof}
Let $f$ be the input of the assumed reflective Borcherds product. 
We conclude from Lemma \ref{lem:principal} that $\eta(\tau)^{12}f(\tau)$ is a holomorphic modular form of weight $7-l/2$ for the Weil representation $\rho_M$. It follows that $l\leq 14$. When $l=14$, $\eta^{12}f$ is of weight $0$, so it is a variant of $\rho_M$, which is impossible because the constant coefficient $c(0,0)$ is nonzero. Therefore, $l\leq 13$.     
\end{proof} 

\begin{lemma}\label{lem:full-group}
Let $M=U\oplus K$ be an even lattice of signature $(l,2)$ with $l\geq 15$. Suppose that $F$ is a reflective Borcherds product on $M$. Then $F$ is unique up to a power. In particular, $F$ is modular for $\Orth^+(M)$. 
\end{lemma}
\begin{proof}
Let $F$ and $F_1$ be two reflective Borcherds products on $M$. We denote their inputs by $f$ and $f_1$, respectively.  There exist nonzero integers $a$ and $b$ such that the coefficient of $q^{-1}e_0$ in $h:=af-bf_1$ is zero. It follows that $\eta^{12}h$ is a holomorphic modular form of weight $7-l/2$. Since the weight $7-l/2$ is negative when $l\geq 15$, we have $h=0$. Therefore, $F^b=F_1^a$. 
\end{proof}

\begin{lemma}\label{lem:overlattice}
Let $\lambda \in U$ satisfying $\lambda^2=2$. Suppose that $M=U\oplus K$ has a reflective Borcherds product vanishing on $\lambda^\perp$. Let $K_1$ be an even overlattice of $K$. Then $U\oplus K_1$ also has a reflective Borcherds product vanishing on $\lambda^\perp$.     
\end{lemma}
\begin{proof}
Let $f$ denote the input of the assumed reflective Borcherds product $F$. We use notations in Lemma \ref{lem:principal}.
Note that $K<K_1<K_1'<K'$. By applying the operator in \cite[Lemma 5.6]{Bru02} to $f$, we obtain a weakly holomorphic modular form of the same weight for the Weil representation attached to $K_1$ as
$$
f|\uparrow_K^{K_1} = \sum_{\gamma \in K_1'/K_1} \sum_{n\in \ZZ-\gamma^2/2} \sum_{x \in (\gamma+K_1)/K} c(x,n)q^n e_\gamma \in M_{1-l/2}^!(\rho_{K_1}).
$$
We denote the Fourier coefficient of $q^n e_\gamma$ in $f|\uparrow_K^{K_1}$ by $c'(\gamma,n)$. Then $c'(0,-1)=c(0,-1)\neq 0$, which yields that $f|\uparrow_K^{K_1}$ is nonzero. Let $c'(\gamma,n)\neq 0$ with $n<0$. Then $n=-1/s$ for some positive integer $s$. For $x\in (\gamma+K_1)/K$ with $c(x,n)\neq 0$, we have $\ord(x)x\in K$ and thus $\ord(x)\gamma \in K_1$, so $s\gamma \in K_1$ by Lemma \ref{lem:principal}. It follows that the order of $\gamma$ in $K_1'/K_1$ is $s$ or $s/2$. Thus every divisor of the Borcherds product $\Borch(f|\uparrow_K^{K_1})$ is reflective. It remains to prove that $\Borch(f|\uparrow_K^{K_1})$ has no poles. 

If $c'(\gamma, -1/s)\neq 0$ and $\gamma$ is of order $s$ in $K_1'/K_1$, then 
$$
c'(\gamma, -1/s) = \sum_{\substack{x \in (\gamma+K_1)/K\\ \ord(x)=s}} c(x,-1/s) > 0.
$$

If $c'(\gamma, -1/(4s))\neq 0$ and $\gamma$ is of order $2s$ in $K_1'/K_1$, then 
$$
c'(\gamma, -1/(4s)) = \sum_{\substack{x \in (\gamma+K_1)/K\\ \ord(x)=4s}} c(x,-1/(4s)) + \sum_{\substack{y \in (\gamma+K_1)/K\\ \ord(y)=2s}} c(y,-1/(4s)). 
$$
In this case, we have 
$$
c'(2\gamma, -1/s) + c'(\gamma, -1/(4s)) \geq \sum_{\substack{y \in (\gamma+K_1)/K\\ \ord(y)=2s}} \Big( c(2y,-1/s) + c(y,-1/(4s)) \Big) \geq 0.  
$$
We have thus proved the lemma.
\end{proof}

\begin{remark}
We give two remarks about Lemma \ref{lem:overlattice}. 
\begin{enumerate}
\item If $F$ does not vanish on $\lambda^\perp$, then $\Borch(f|\uparrow_K^{K_1})$ may be zero, for example $K=U\oplus D_8$, $K_1=U\oplus E_8$ and $F$ is the unique reflective Borcherds product of weight $4$ on $2U\oplus D_8$. If $\Borch(f|\uparrow_K^{K_1})$ is nonzero, then it is a reflective Borcherds product on $U\oplus K_1$. 
\item Let $F$ be a reflective modular form for $\Orth^+(M)$. For an even overlattice $M_1$ of $M$, the product of $F|g$ for all $g \in \Orth^+(M_1) / \Orth^+(M)$ may be not reflective on $M_1$, because $\Orth^+(M)<\Orth^+(M_1)$ does not hold in general (see the example in (1)). 
\end{enumerate}
\end{remark}

\begin{theorem}\label{th:l=22}
Let $M=U\oplus K$ be an even lattice of signature $(22,2)$ which has a reflective Borcherds product. Then $M$ is isomorphic to $2U\oplus D_{20}$.    
\end{theorem}

\begin{proof}
We first assume that the discriminant group of $M$ has length (i.e. the minimal number of generators of $M'/M$) $l(M)\leq 3$. By Nikulin's result\cite{Nik80}, we can write $M=2U\oplus 2E_8\oplus L_4$, where $L_4$ is an even positive definite lattice of rank $4$. Let $F$ be a reflective Borcherds product on $M$. By Lemma \ref{lem:full-group}, $F$ is modular for $\Orth^+(M)$. By applying Theorem \ref{th:fake-root} to $F$, we find that the root system associated with $F$ is of type $2E_8\oplus R$, where $R$ is a root system of rank $4$ related to $L_4$. We use notations fixed in Theorem \ref{th:fake-root}. By \cite[Proposition 9.6]{Wan19}, the weight of $F$ is $24\beta_0$ and $R$ contains only roots $v$ with $v^2=2$ or $v^2=1$ in $L_4'$, where $\beta_0:=f(-1,0)$. Let $r_2$ denote the number of $2$-roots in $R$ and $r_1$ denote the number of $1$-roots in $R$ (counting multiplicity), i.e. 
$$
r_1 = \sum_{\ell \in L_4',\; \ell^2=1} f(0,\ell).
$$
We derive from \eqref{eq1} and \eqref{eq2} that 
$$
\frac{480\beta_0+\beta_0 r_2 + r_1 + 48\beta_0}{24} -\beta_0 = \frac{1}{40}\left( 480\beta_0 \times 2 + 2\beta_0 r_2 + r_1 \right) = 30\beta_0,
$$
which yields that $r_2=24$ and $r_1=192\beta_0$. Since $R$ contains exactly $24$ roots of norm $2$, it is easy to show that $R$ is $F_4$ or $B_4$. 

When $R=B_4$, we conclude from Part (4) of Theorem \ref{th:fake-root} that $D_4<L_4<\ZZ^4$ and thus $L_4=D_4$.  

When $R=F_4$, we conclude directly from Part (4) of Theorem \ref{th:fake-root} that $L_4=D_4$. 

We have thus proved that $M\cong 2U\oplus 2E_8\oplus D_4\cong 2U\oplus D_{20}$. Since $F$ is modular for the full orthogonal group, the associated root system is of type $2E_8\oplus F_4$. By \eqref{eqA} we have
$$
\frac{480\beta_0+24\beta_0+24b+48\beta_0}{24} - \beta_0 = 30\beta_0 = 6\beta_0 + 3b,
$$
which follows that $b=8\beta_0$. Thus the fake root system of $F$ is $(E_8,\beta_0;1)^2\oplus (F_4,\beta_0,8\beta_0;1)$, which coincides with Example \ref{example} (7) when $\beta_0=1$. 

We now consider the general case. Suppose that $M=U\oplus K$ has length $l(M)>3$. Let $K_1$ be a maximal even overlattice of $K$. By Lemma \ref{lem:sym-bound} and Lemma \ref{lem:overlattice}, $M_1:=U\oplus K_1$ also has a reflective Borcherds product. Note that $M_1$ is maximal and then $l(M_1)\leq 3$. We know from the above discussions that $M_1\cong 2U\oplus D_{20}$. Let $p$ be a prime factor of the order of the quotient group $M_1/M$. Then there exists an even overlattice $M_0$ of $M$ such that $M<M_0<M_1$ and $M_1/M_0\cong \ZZ/p\ZZ$. If $p$ is odd, then $M_0$ has length $l(M_0)\leq 2$, which yields that $M_0\cong 2U\oplus D_{20}$, a contradiction. If $p=2$, then either $M_0$ has length $l(M_0)\leq 3$ or $M_0'/M_0=(\ZZ/2\ZZ)^4$. In the former case, we have $M_0\cong 2U\oplus D_{20}$, a contradiction. In the latter case, $M_0$ is 2-elementary. By Nikulin's classification \cite[Theorem 3.6.2]{Nik80}, $M_0$ is isomorphic to $2U\oplus E_8\oplus D_8\oplus D_4$ or $2U\oplus E_8\oplus E_7\oplus A_1\oplus D_4$. We exclude the former case by \cite[Theorem 1.1]{Wan22} and the latter case by Theorem \ref{th:fake-root}, because the associated root system is of type $E_8\oplus E_7\oplus R$ and the Coxter numbers of $E_8$ and $E_7$ are distinct. Therefore, there is no lattice of type $U\oplus K$ with length $l(M)>3$ which has a reflective Borcherds product. We then complete the proof. 
\end{proof}

\begin{remark}\label{rem:l=22}
For $M=2U\oplus D_{20}$, through the identification 
$$
\Orth^+(M) = \Orth^+(M')=\Orth^+(M'(2)),
$$
we can view the Borcherds form $\Psi_{24}$ (see Example \ref{example} (7)) as a reflective modular form on the lattice $2U(2)\oplus D_{20}'(2)$. By Bruinier's result \cite[Theorem 1.4]{Bru14}, this form is a Borcherds product on $2U(2)\oplus D_{20}'(2)$. Thus Theorem \ref{th:l=22} does not hold when $M$ does not splits $U$. This phenomena occurs because if $M$ is not of type $U\oplus K$ then the principal part of the input of a reflective Borcherds product on $M$ may have Fourier coefficients corresponding to empty quadratic divisors, in particular, Lemma \ref{lem:principal} does not hold in this case. We also conclude that if an even lattice has a reflective modular form then its even overlattice may not have a reflective modular form, for example $2U(2)\oplus D_{20}'(2)$ and its even overlattice $2U\oplus D_{20}'(2)$.  
\end{remark}

\begin{theorem}\label{th:l=21}
There is no even lattice of signature $(21,2)$ with a reflective modular form.     
\end{theorem}
\begin{proof}
We first prove that there is no maximal even lattice of signature $(21,2)$ with a reflective modular form.  Suppose that $M$ is a maximal even lattice with a reflective modular form $F$ of weight $k$ for $\Orth^+(M)$.  By \cite[Lemma 2.2]{Wan23a}, we can represent $M=2U\oplus 2E_8\oplus L_3$ for some maximal even positive definite lattice of rank $3$.  Then $F$ is a reflective Borcherds product and we can apply Theorem \ref{th:fake-root} to $F$. Clearly, the root system associated with $F$ is of type $2E_8\oplus R$, where $R$ is a root system of rank $3$ contained in $L_3'$. There are eight possibilities of $R$. We discuss by cases. We set $a_0=f(-1,0)$ for simplicity.
\begin{enumerate}
\item $R=A_3$. By \eqref{eqA} we have 
$$
\frac{480a_0+12a+2k}{24} - a_0 = 30a_0 = \frac{4a}{d},
$$
which yields that 
$$
a=15a_0d/2 \quad \text{and} \quad k=a_0(132-45d). 
$$
Since $k>0$, $d=1$ or $2$. When $d=1$, $a=15a_0/2\neq a_0$, a contradiction. When $d=2$, by Theorem \ref{th:fake-root} (4) we have the bound $A_3(2)<L_3<A_3'(2)$, which yields that $L_3=3A_1$ because $L_3$ is maximal. This leads to a contradiction by Lemma \ref{lem:non-reflective}.

\item $R=B_3$. By \eqref{eqA} we have 
$$
\frac{480a_0+12a+6b+2k}{24} - a_0 = 30a_0 = \frac{4a+b}{d},
$$
which yields that
$$
4a+b=30a_0d \quad \text{and} \quad 6a+3b+k=132a_0. 
$$
We deduce from $k+3b/2=(132-45d)a_0>0$ that $d=1$ or $2$. 

When $d=1$, by Theorem \ref{th:fake-root} we have $A_3<L_3<\ZZ^3$, which forces that $L_3=A_3$, i.e. $M=2U\oplus 2E_8\oplus A_3$. This leads to a contradiction by Lemma \ref{lem:non-reflective}.

When $d=2$, we have the bound $A_3(2)<L_3<3A_1$. Since $L_3$ is maximal, $L_3=3A_1$ and thus $M=2U\oplus 2E_8\oplus 3A_1$, a contradiction by Lemma \ref{lem:non-reflective}. 

\item $R=C_3$. When $c\neq 0$, we have $L_3=3A_1(d)$. We know from \cite[Theorem 8.1] {Wan19} that $2U\oplus 2E_8\oplus A_1(m)$ has a reflective modular form if and only if $m=1$ or $2$. Then \cite[Lemma 5.2]{Wan19} yields that $d=1$ or $2$. When $d=2$, $L_3$ is not maximal, a contradiction. When $d=1$, $M=2U\oplus 2E_8\oplus 3A_1$, a contradiction by Lemma \ref{lem:non-reflective}.  

We now assume that $c=0$. By \eqref{eqA} we have
$$
\frac{480a_0+6a+12b+2k}{24}-a_0 = 30a_0 = \frac{2a+2b}{d},
$$
which yields that 
$$
a+b=15a_0d \quad \text{and} \quad a+2b+k/3=44a_0.
$$
We derive from $b+k/3=(44-15d)a_0>0$ that $d=1$ or $2$. 

When $d=1$, we have the bound $3A_1<L_3<A_3'(2)$, therefore, $L_3=3A_1$ and $M=2U\oplus 2E_8\oplus 3A_1$, a contradiction.

When $d=2$, we have the bound $3A_1(2)<L_3<A_3'(4)$, which forces that $L_3=3A_1(2)$. This contradicts the fact that $L_3$ is maximal.  

\item $R=A_1\oplus A_2$, $A_1\oplus C_2$, $A_1\oplus G_2$ or $3A_1$ with $c\neq 0$ for $A_1$ or $C_2$.  We see from Theorem \ref{th:fake-root} that $L_3$ can be expressed as $A_1(d)\oplus L_2$. By \cite[Lemma 5.2]{Wan19}, $2U\oplus 2E_8\oplus L_2$ has a reflective modular form. We then conclude from Lemma \ref{lem:L2} that $L_3=A_1(d)\oplus A_2$ or $A_1(d)\oplus 2A_1$. As at the beginning of Case (3), \cite[Theorem 8.1, Lemma 5.2]{Wan19} yields that $d=1$ or $2$. This contradicts Lemma \ref{lem:non-reflective}. 

\item $R=A_1\oplus A_2$ with $c=0$. By \eqref{eqA} we find
$$
\frac{480a_0+2a_1+6a_2+2k}{24}-a_0=30a_0=\frac{2a_1}{d_1}=\frac{3a_2}{d_2},
$$
which follows that
$$
a_1+3a_2+k=132a_0 \quad \text{and} \quad a_1=15a_0d_1 \quad \text{and} \quad a_2=10a_0d_2.
$$
Clearly, $d_1>1$ and $d_2>1$.
We further conclude from Theorem \ref{th:fake-root} (3) that $k\geq 8a_0+a_1/2+a_2$, more precisely, this bound follows from the singular weights of holomorphic Jacobi forms of indices $E_8$, $A_1$ and $A_2$. There is no $(d_1,d_2)$ satisfying these constrains. 

\item $R=A_1\oplus C_2$. By \eqref{eqA} we have
$$
\frac{480a_0+2a_1+4a_2+4b_2+2k}{24} - a_0 =30a_0 = \frac{2a_1}{d_1} = \frac{2a_2+b_2}{d_2},
$$
which yields that 
$$
a_1+2a_2+2b_2+k=132a_0 \quad \text{and} \quad a_1=15a_0d_1 \quad \text{and} \quad 2a_2+b_2=30a_0d_2. 
$$
Clearly, $d_1>1$.
We see from Theorem \ref{th:fake-root} (3) that $k> 8a_0+a_1/2$. When $d_2=1$, $a_2=a_0$ and thus $b_2=28a_0$. Notice that $k+b_2=132a_0-15a_0d_1-30a_0d_2$. These constrains yield only two solutions $(d_1,d_2)=(2,1)$ and $(d_1,d_2)=(2,2)$. Recall that $L_3$ is maximal. We see form the bound $A_1(d_1)\oplus 2A_1(d_2) < L_3 < \ZZ(d_1/2)\oplus \ZZ^2(d_2)$ that $L_3$ has to be $A_3$. This leads to a contradiction by Lemma \ref{lem:non-reflective}.

\item $R=A_1\oplus G_2$. By Theorem \ref{th:fake-root} (4), $L_3$ can be represented as $A_1(m)\oplus A_2(d)$. Lemma \ref{lem:L2} further yields that $L_3=A_1\oplus A_2$ or $A_1(2)\oplus A_2$, a contradiction by Lemma \ref{lem:non-reflective}.

\item $R=3A_1$. By \eqref{eqA} we obtain
$$
\frac{480a_0+2a_1+2a_2+2a_3+2k}{24} - a_0 = 30a_0 = \frac{2a_1}{d_1} = \frac{2a_2}{d_2} =\frac{2a_3}{d_3},
$$
which yields that
$$
a_1+a_2+a_3+k=132a_0 \quad \text{and} \quad a_i=15a_0d_i, \; i=1,2,3.
$$
Thus $d_i>1$. By Theorem \ref{th:fake-root} (3), $k\geq 8a_0+(a_1+a_2+a_3)/2$. There is no $(d_1,d_2,d_3)$ satisfying these constrains. 
\end{enumerate}
We now consider the general case. Suppose that $M$ is an even lattice of signature $(21,2)$ with a reflective modular form $F$ for some finite-index subgroup $\Gamma < \Orth^+(M)$. By \cite[Corollory 3.2]{Ma18}, there exists a lattice $M_1$ on $M\otimes\QQ$ such that $\Orth^+(M) \subset \Orth^+(M_1)$ and that $M_1$ is a scaling of an even lattice of type $2U\oplus L$. Then the symmetrization 
$
\prod_{g \in \Orth^+(2U\oplus L) / \Gamma} F|g
$
defines a reflective modular form for $\Orth^+(2U\oplus L)$. Let $M_2$ be a maximal even overlattice of $2U\oplus L$. By Lemmas \ref{lem:sym-bound} and \ref{lem:overlattice}, $M_2$ also has a reflective modular form. This contradicts the above particular case. We thus prove the theorem. 
\end{proof}

\begin{lemma}\label{lem:non-reflective}
There is no reflective Borcherds product on $2U\oplus 2E_8\oplus L$ when $L$ takes one of the following lattices:
$$
3A_1,\quad A_3, \quad A_1\oplus A_2, \quad A_1(2)\oplus A_2, \quad A_2(2), \quad A_2(3), \quad 2A_1(2), \quad A_1\oplus A_1(2).
$$
\end{lemma}
\begin{proof}
The lemma may be verified by the obstruction principle \cite[Theorem 3.1]{Bor99}, because the discriminant groups of these lattices are simple.  Some cases can also be proved by Theorem \ref{th:fake-root}.  

We see from Theorem \ref{th:fake-root} and
$$
2U\oplus 2E_8\oplus 3A_1 \cong 2U\oplus E_8 \oplus E_7 \oplus D_4,
$$
that there is no reflective Borcherds product on $2U\oplus 2E_8\oplus 3A_1$, because the irreducible components $E_8$ and $E_7$ have distinct  Coxeter numbers. 

We claim that $2U\oplus 2E_8\oplus A_1\oplus A_2$ has no reflective Borcherds product. Otherwise, the associated root system is of type $2E_8\oplus A_1\oplus G_2$ with $d=1$ and by applying Theorem \ref{th:fake-root} we obtain 
$$
\frac{480a + 2a+2c+ 6a + 6b + 2k}{24} - a = 30a =2a+c/2 = 3a+b,
$$
which yields that $c=56a$, $b=27a$ and thus $k=-9a<0$, a contradiction.  

We claim that $2U\oplus 2E_8\oplus A_1(2)\oplus A_2$ has no reflective Borcherds product. Otherwise, the associated root system is of type $2E_8\oplus A_1(1/2)\oplus G_2$ and by applying Theorem \ref{th:fake-root} we obtain 
$$
\frac{480a + 2a'+2c+ 6a + 6b + 2k}{24} - a = 30a =\frac{2a'+c/2}{2} = 3a+b,
$$
which yields that $b=27a$. By considering the quasi pullback to $2U\oplus 2E_8\oplus A_1(2)$, we see from Example \ref{example} (6) that $a'=14a$ and $c=64a$. Therefore, $k=-30a<0$, a contradiction. In a similar way, we can show that $2U\oplus 2E_8\oplus A_1\oplus A_1(2)$ has no reflective Borcherds product. 
\end{proof}

\begin{lemma}\label{lem:L2}
Let $L_2$ be an even positive definite lattice of rank $2$. If $2U\oplus 2E_8\oplus L_2$ has a reflective modular form, then $L_2$ is isomorphic to $A_2$ or $2A_1$.    
\end{lemma}

\begin{proof}
Let $M=2U\oplus 2E_8\oplus L_2$ and $F$ be a reflective modular form of weight $k$ for $\Orth^+(M)$. By Theorem \ref{th:fake-root}, the root system associated with $F$ is of type $2E_8\oplus R$, where $R$ is a rank-two root system contained in $L_2'$. There are four possibilities of $R$. We discuss by cases. We set $a_0=f(-1,0)$ for simplicity. 
\begin{enumerate}
\item $R=G_2$. Theorem \ref{th:fake-root} (4) yields that $L_2=A_2(d)$. By \eqref{eqA} we have
$$
\frac{480a_0+6a+6b+2k}{24} - a_0 = 30a_0 = \frac{3a+b}{d},
$$
which implies that
$$
3a+b=30a_0d \quad \text{and} \quad a+b+k/3 = 44a_0.
$$
We deduce from $k+2b=(132-30d)a_0>0$ that $d\leq 4$. We calculate by Theorem \ref{th:fake-root} (2) that the norm of the Weyl vector is positive when $d=4$, a contradiction. Therefore, $d\leq 3$. When $d=1$, $M$ does have a reflective Borcherds product (see Example \ref{example} (4)). When $d=2$ or $3$, we know from Lemma \ref{lem:non-reflective} that $M$ has no reflective Borcherds product.

\item $R=A_2$. By \eqref{eqA} we have
$$
\frac{480a_0+6a+2k}{24} - a_0 = 30a_0 = \frac{3a}{d},
$$
which yields that
$$
a=10a_0d \quad \text{and} \quad k=(132-30d)a_0.
$$
We see from Theorem \ref{th:fake-root} (3) that $k\geq 8a_0+a$. These constrains yield that $d\leq 3$. When $d=1$, $a=10a_0\neq a_0$, a contradiction. When $d=2$, the lattice $L_2$ is bounded by $A_2(2)<L_2<A_2'(2)$, which forces that $L_2=A_2(2)$, a contradiction by Lemma \ref{lem:non-reflective}. When $d=3$, we have $A_2(3)<L_2<A_2'(3)\cong A_2$, which forces that $L_2=A_2(3)$, a contradiction by Lemma \ref{lem:non-reflective}. 

\item $R=C_2$. If $c\neq 0$, then $L_2=2A_1(d)$. We know from \cite[Theorem 8.1] {Wan19} that $2U\oplus 2E_8\oplus A_1(m)$ has a reflective modular form if and only if $m=1$ or $2$. By \cite[Lemma 5.2]{Wan19}, $d=1$ or $2$. When $d=1$, $M$ does have a reflective modular form (see Example \ref{example} (3)). When $d=2$, $M=2U\oplus 2E_8\oplus 2A_1(2)$, a contradiction by Lemma \ref{lem:non-reflective}.

We now assume that $c=0$. By \eqref{eqA} we have 
$$
\frac{480a_0+4a+4b+2k}{24} - a_0 = 30a_0 = \frac{2a+b}{d},
$$
which yields that
$$
2a+b=30a_0d \quad \text{and} \quad a+b+k/2 = 66a_0.
$$
We derive from $b+k=(132-30d)a_0>0$ that $d\leq 4$. The lattice $L_2$ is bounded by $2A_1(d)<L_2<\ZZ^2(d)$. When $d=1$, it forces that $L_2=2A_1$. When $d=2$, $L_2=2A_1(2)$ or $2A_1$. When $d=3$, $L_2=2A_1(3)$. When $d=4$, $L_2=2A_1(4)$ or $2A_1(2)$. We prove that $L_2=2A_1$ by Lemma \ref{lem:non-reflective} and the argument at the beginning of this case. 

\item $R=2A_1$. If $c\neq 0$ for some $A_1$, then $L_2$ has the form $A_1(m)\oplus A_1(n)$. The argument at the beginning of Case (3) yields that $m$ and $n$ can only be $1$ or $2$. By Lemma \ref{lem:non-reflective}, $L_2=2A_1$. 

We now assume that $c\neq 0$ for each of copies of $A_1$. By \eqref{eqA} we have
$$
\frac{480a_0+2a_1+2a_2+2k}{24}-a_0 = 30a_0 = \frac{2a_1}{d_1} = \frac{2a_2}{d_2},
$$
which yields that
$$
a_1+a_2+k=132a_0 \quad \text{and} \quad a_1=15a_0d_1 \quad \text{and} \quad a_2=15a_0d_2. 
$$
We assume that $d_1\geq d_2$ for convenience. When $d_2=1$, $a_2=15a_0\neq a_0$, a contradiction. Therefore, $d_1\geq d_2\geq 2$. We derive from Theorem \ref{th:fake-root} (3) that $k\geq 8a_0+(a_1+a_2)/2$. It follows that $(d_1,d_2)=(2,2)$ or $(d_1,d_2)=(3,2)$. By Theorem \ref{th:fake-root} (4), $L_2$ is bounded by $A_1(d_1)\oplus A_1(d_2)<L_2<\ZZ(d_1/2)\oplus \ZZ(d_2/2)$. It forces that $L_2=2A_1$, $2A_1(2)$ or $A_1(2)\oplus A_1(3)$. By Lemma \ref{lem:non-reflective} and the argument at the beginning of Case (3), $L_2=2A_1$.
\end{enumerate}
The proof is complete. 
\end{proof}

\begin{proof}[Proof of Theorem \ref{MTH2}]
The case $l>26$ was proved in \cite[Theorem 4.7 (1)]{Wan18} and the case $23\leq l \leq 25$ was proved in \cite[Theorem 4.11]{Wan18}.  The cases $l=22$ and $l=21$ were proved in Theorems \ref{th:l=22} and \ref{th:l=21}, respectively.  The case $l=26$ can be proved by combining Lemma \ref{lem:overlattice} and the argument in the proof of \cite[Theorem 4.7 (2)]{Wan18}. 
\end{proof}

\begin{remark}\label{rem:l=26}
\cite[Theorem 4.7 (2)]{Wan18} should be corrected to Statement (2) of Theorem \ref{MTH2}. Let $p$ be a prime number. As in Remark \ref{rem:l=22}, through the identification $\Orth^+(\II_{26,2})=\Orth^+(\II_{26,2}(p))$, we can view the Borcherds form $\Phi_{12}$ as a reflective modular form on $\Orth^+(\II_{26,2}(p))$. \cite[Theorem 1.4]{Bru14} then yields that $\Phi_{12}$ can be constructed as a Borcherds product on $\II_{26,2}(p)$. Therefore, $\II_{26,2}$ is not the unique lattice of signature $(26,2)$ with a reflective Borcherds product. 
\end{remark}

\section{A proof of Theorem \ref{MTH3}}\label{sec:non-free}
Let $M$ be an even lattice of signature $(l,2)$ and $\Gamma$ be a finite-index subgroup of $\Orth^+(M)$ acting on the symmetric domain $\cD(M)$. By analyzing the smoothness of the Satake--Baily--Borel compactification of the quotient $\cD(M)/\Gamma$, in 2017 Shvartsman and Vinberg \cite{SV17} proved the algebra of modular forms for $\Gamma$ is never free when $l>10$. In this section, we give a new proof of their result by the classification of reflective modular forms. 

\begin{theorem}
Let $M$ be an even lattice of signature $(l,2)$ with $l\geq 11$ and $\Gamma$ be a finite-index subgroup of $\Orth^+(M)$. 
Then the algebra of modular forms of integral weight and trivial character for $\Gamma$ is never freely generated by $l+1$ modular forms. In particular, the Satake--Baily--Borel compactification of $\cD(M)/\Gamma$ is never isomorphic to a weighted projective space.
\end{theorem}
\begin{proof}
Suppose that the algebra of modular forms for $\Gamma$ is freely generated by $l+1$ modular forms.  By \cite[Theorem 3.5]{Wan21},  the modular Jacobian $J$ of the $l+1$ generators is a cusp form of weight $k$ and the determinant character for $\Gamma$ which vanishes precisely with multiplicity one on all quadratic divisors $\lambda^\perp$ for $\lambda \in M$ such that $\sigma_\lambda \in \Gamma$. The weight $k$ equals $l$ plus the sum of the weights of the $l+1$ generators. Moreover, there exist positive norm vectors $\lambda_i \in M$ for $1\leq i\leq s$ such that $J^2$ can be decomposed as $J^2=\prod_{i=1}^s J_i$, where each $J_i$ is a modular form of trivial character for $\Gamma$ and vanishes precisely with multiplicity two on the orbit $\Gamma\cdot \lambda_i^\perp$.  

By \cite[Corollory 3.2]{Ma18}, there exists a lattice $M_1$ on $M\otimes\QQ$ such that $\Orth^+(M) \subset \Orth^+(M_1)$ and that $M_1$ is a scaling of an even lattice of type $M_2=2U\oplus L$. For each $i$ we define 
$$
\widetilde{J}_i = \prod_{g \in \Orth^+(M_2) / \Gamma} J_i|g
$$
which is a modular form for $\Orth^+(M_2)$ and vanishes precisely on the orbit $\Orth^+(M_2)\cdot \lambda_i^\perp$ with some multiplicity $m_i$. By Bruinier's result \cite{Bru02}, $\widetilde{J}_i$ is a Borcherds product on $M_2$. We denote the input of $\widetilde{J}_i$ by $\tilde{f}_i$. Then the Borcherds product of $\tilde{f}_i/m_i$ defines a reflective modular form for $\Orth^+(M_2)$ vanishing precisely with multiplicity one on $\Orth^+(M_2)\cdot \lambda_i^\perp$. Let $F_i$ denote this product and $f_i$ denote its input.  We denote by $F$ the product of all distinct $F_i$. Then $F/J$ is a holomorphic modular form for $\Gamma$. It follows that $F$ is a cusp form and its weight is at least $k$. Recall that $k\geq l+(l/2-1)(l+1)$. 

For convenience we assume that every $\lambda_i$ is a primitive vector of $M_2'$. Then there exists a positive integer $d_i$ such that $\lambda_i^2=2/d_i$ and $d_i\lambda_i \in M_2$. The principal part of $f_i$ consists of Fourier coefficients of type $q^{-1/d_i}e_x$ or  $q^{-1/d_i}e_x - q^{-1/(4d_i)}e_{x/2}$ by Lemma \ref{lem:principal}.

We claim that $d_i\leq 2$ for all $i$. Suppose that there is some $d_i\geq 3$. Then $\eta^{8}f_i$ would define a holomorphic modular form of weight $5-l/2<0$ for the Weil representation $\rho_{M_2}$, a contradiction. 

Suppose that there is some $d_i=1$. We conclude from \cite[Theorem 1.2]{Wan23a} that $M_2=2U\oplus 2E_8$ or $\II_{26,2}$. The above $F$ has weight $132$ or $12$, respectively. The weight of $F$ is too small, contradicting the inequality $k\geq l+(l/2-1)(l+1)$.

Therefore, every $d_i$ is $2$. Let $f$ be the input of $F$. Then $\eta^{12}f$ defines a holomorphic modular form of weight $7-l/2$ for the Weil representation, which yields that $l\leq 13$. Thus $L$ has rank $9$, $10$ or $11$. The fake root system $\hat{\mathcal{R}}$ associated with $F$ is of type $ADE$ with $d=2$. More precisely, its irreducible components have only the following possibilities:
$$
(A_1,1|-1;2), \quad (A_1,1|0;2), \quad (A_n,1;2), n\geq 2, \quad (D_n,1;2), n\geq 4, \quad (E_n,1;2), n=6,7,8. 
$$
Since $\hat{\mathcal{R}}$ has rank $9$, $10$ or $11$ and its irreducible components have the same fake Coxeter number, it has only the following possibilities:
\begin{align*}
&(A_1,1|-1;2)^{n}, \quad (A_1,1|0;2)^n, \quad \text{for $n=9,10,11$,}\\
& (D_n,1;2), \quad (A_n,1;2), \quad \text{for $n=9,10,11$,} \\
& (D_5,1;2)^2, \quad (A_5,1;2)^2, \quad (D_4,1;2)\oplus (A_5,1;2). 
\end{align*}
We derive from \eqref{eqA} that the weight $k=9$ for the first possibility and $k=(6-n/2)h$ for the other possibilities, where $n=l-2$ and $h$ are respectively the rank and the  Coxeter number of the root system $\mathcal{R}$. For example, for the last possibility $n=9$, $h=6$ and $k=9$. The weight $k$ is too small, which contradicts the inequality $k\geq l+(l/2-1)(l+1)$. The proof is complete. 
\end{proof}

\section{Reflective hyperbolic lattices and a proof of Theorem \ref{MTH4}}\label{sec:finiteness}
In this section we review a connection between reflective modular forms and reflective hyperbolic lattices and apply it to the proof of Theorem \ref{MTH4}. 

\subsection{Reflective hyperbolic lattices}\label{subsec:Nikulin}
Let $S$ be an even lattice of signature $(s,1)$ with $s\geq 2$. Such $S$ is called hyperbolic. We define the half-cone $V^+(S)$ as a connected component of the space 
$$
V(S)=\{ x \in S\otimes \RR : x^2 <0 \}.
$$
Let $\Orth^+(S)$ denote the index-two subgroup of $\Orth(S)$ that fixes $V^+(S)$. Let $W$ be the subgroup of $\Orth^+(S)$ generated by reflections and $\mathcal{M}$ be an associated fundamental polyhedron. We denote by $A(\mathcal{M})$ the subgroup of $\Orth^+(S)$ that fixes $\mathcal{M}$. The semi-direct product of $W$ by $A(\mathcal{M})$ is a subgroup of $\Orth^+(S)$. If $A(\mathcal{M})$ has finite index in the quotient group $\Orth^+(S) / W$ then $S$ is called \textit{reflective}. A reflective hyperbolic lattice $S$ is called \textit{elliptic} if $A(\mathcal{M})$ is finite, otherwise it is called \textit{parabolic}. 

In 1980 Nikulin \cite{Nik80b} proved that the set of elliptic reflective hyperbolic lattices of fixed signature is finite up to scaling. In 1984 Vinberg \cite{Vin84} showed that there is no elliptic reflective hyperbolic lattice of signature $(s,1)$ when $s\geq 30$. In 1996 Esselmann \cite{Ess96} improved Vinberg's bound to $s\geq 22$. Later, Nikulin \cite[Theorem 1.1.3]{Nik96} proved that the set of parabolic reflective hyperbolic lattices of fixed signature is finite up to scaling. As noted by Nikulin, Vinberg's method yields that there is no parabolic reflective hyperbolic lattice of signature $(s,1)$ when $s\geq 43$. However, the expected sharp bound is $s\geq 26$. 

\subsection{Arithmetic mirror symmetry conjecture}
There is a connection between reflective modular forms and reflective hyperbolic lattices. 

In 1998 Borcherds \cite[Theorem 12.1]{Bor98} proved that if $U\oplus S$ has a reflective Borcherds product with a nonzero Weyl vector then the hyperbolic lattice $S$ is reflective. Moreover, if the Weyl vector has negative norm, then $S$ is of elliptic type. Later, he \cite{Bor00} applied this criterion to the Borcherds product of Example \ref{example} (7) and proved that $U\oplus D_{20}$ is elliptic reflective. By Esselmann's result \cite{Ess96}, this gives an example of elliptic reflective hyperbolic lattices of largest rank.   

In the same year, Gritsenko and Nikulin \cite[Conjecture 2.2.4]{GN98a} proposed the following arithmetic mirror symmetry conjecture to describe this connection in the more general case. 

\begin{conjecture}[Arithmetic Mirror Symmetry Conjecture]\label{conj}
\noindent
\begin{enumerate}
    \item[(i)] Let $M$ be an even lattice of signature $(l,2)$ with $l\geq 3$ and $c$ be a primitive norm zero vector of $M$. Suppose that there is a reflective modular form for $\Orth^+(M)$ which vanishes on some quadratic divisor $r^\perp$ for $r \in c_M^\perp / \ZZ c$, where $c_M^\perp=\{x\in M: (x,c)=0\}$. Then the hyperbolic even lattice $c_M^\perp / \ZZ c$ is reflective. 
    \item[(ii)] Any reflective hyperbolic even lattice $S$ may be obtained by the above construction.  
\end{enumerate}
\end{conjecture}

Part (i) of the conjecture was proved by Looijenga \cite[Corollary 5.1]{Loo03} in 2003. We will give an explicit counterexample to Part (ii) later.

\subsection{A proof of Theorem \ref{MTH4}} Part (3) of Theorem \ref{MTH4} has been proven in Lemma \ref{lem:sym-bound}. We here prove the first two parts. 

\begin{theorem}
\noindent
\begin{enumerate}
    \item The set of even lattices of type $2U\oplus L$ with a reflective modular form is finite. 
    \item The set of even lattices of signature $(l,2)$ and type $U\oplus K$ which has a reflective Borcherds product vanishing on $r^\perp$ is finite when $l\geq 7$, where $r\in U$ and $r^2=2$. 
\end{enumerate}    
\end{theorem}
\begin{proof}
Suppose that there is a reflective modular form for $\Orth^+(2U\oplus L)$. Applying Part (i) of Conjecture \ref{conj} to $2U\oplus L$, we find that $U\oplus L$ is reflective hyperbolic. This claim also follows from Borcherds' theorem reviewed in the previous subsection, since the Weyl vector is clearly nonzero by Theorem \ref{th:fake-root} (2). By Nikulin's result in \S \ref{subsec:Nikulin}, the set of such $U\oplus L$ is finite. This proves (1). 

We now prove (2). Suppose that $M=U\oplus K$ satisfies the conditions of (2). By \cite[Lemma 4.8]{Ma17} there exists an even overlattice $M_1$ of $M$ with length $l(M_1)\leq 4$ and exponent $e(M_1)=e(M)$ or $e(M)/2$.  Recall that the length of an even lattice $T$ is the minimal number of generators of the discriminant group $T'/T$, and the exponent of $T$ is the maximal order of elements of $T'/T$. By Lemma \ref{lem:overlattice}, $M_1$ has a reflective Borcherds product. When $l\geq 7$, Nikulin's result \cite{Nik80} yields that $M_1$ splits as $2U\oplus L$. We conclude from (1) that $e(M_1)$ is bounded from above. Therefore, $e(M)$ is bounded from above too. This leads to the finiteness of $M$. 
\end{proof}

\subsection{An counterexample to Part (ii) of the Gritsenko--Nikulin conjecture}
We know from \cite[Page 492]{GN18} that $S=U\oplus E_8\oplus E_7$ is a reflective hyperbolic lattice of elliptic type. We claim that it cannot be obtained by Part (i) of Conjecture \ref{conj}. 

Suppose that we can obtain $S$ from an even lattice $M$ of signature $(17,2)$ at some primitive norm zero vector $c\in M$. Let $n$ be the positive generator of the ideal 
$$
(c,M)=\{ (c,v): v\in M \}.
$$
Then $c/n \in M'$. We choose $c'\in M$ such that $(c,c')=n$ and
$$
S=\{ v\in M : (v,c)=(v,c')=0 \}.
$$
Let $K$ be the lattice generated by $c$ and $c'$ over $\ZZ$. Then $K\oplus S < M < M' < K'\oplus S'$. It follows that the discriminant group of $M$ has length $l(M)\leq 3$. Thus $M$ splits $2U$. By assumption, there is a reflective modular form for $\Orth^+(M)$. By \cite{Bru02}, this form is a Borcherds product. It is easy to check that the overlattice of $M$ generated by $c/n$ is of type $2U\oplus E_8\oplus E_7$. Lemmas \ref{lem:sym-bound} and \ref{lem:overlattice} together yield that $2U\oplus E_8\oplus E_7$ has a reflective Borcherds product. This contradicts Theorem \ref{th:fake-root}, because the Coxeter numbers of $E_8$ and $E_7$ are different. 

\bigskip

\noindent
\textbf{Acknowledgements} 
The author is supported by the Institute for Basic Science (IBS-R003-D1). 
The author thanks Brandon Williams for verifying Lemma \ref{lem:non-reflective} with his algorithm. 

\bibliographystyle{plainnat}
\bibliofont
\bibliography{refs}

\end{document}